\documentclass[a4paper,12pt]{article}
\usepackage[top=3cm,bottom=2cm,left=2cm,right=2cm]{geometry}
\usepackage{graphics}
\usepackage{graphicx}
\usepackage{amssymb}
\usepackage{amsmath}
\usepackage{amsthm}
\usepackage{dsfont}
\usepackage{marvosym}
\usepackage{multicol}
\usepackage{wasysym}
\usepackage[english]{babel}
\usepackage{cmap}
\usepackage[utf8]{inputenc}
\usepackage[T1]{fontenc}
\usepackage{hrlatex}
\usepackage{setspace}
\newtheorem{theorem}{Theorem}
\newcommand{\Q}{\mathbb{Q}}

\usepackage{setspace}
\makeatletter
 
  \@addtoreset{equation}{section}
\makeatother
\makeatletter
\def\tbcaption{\def\@captype{table}\caption}
\makeatother
\newtheorem{thm}{Theorem}[section]

\newtheorem{lem}[thm]{Lemma}
\newtheorem{prop}[thm]{Proposition}

\newtheorem{conj}[thm]{Conjecture}
\newtheorem{ass}[thm]{Assumption}

\newcommand{\beq}[1]{\begin{equation}\label{#1}}
\newcommand{\eeq}{\end{equation}}

\newcommand{\bZ}{\ensuremath{\mathbb{Z}}}

\begin{document}
\title{The extension of the $D(-k)$-pair $\{k,k+1\}$ to a quadruple}

\author{Nikola Ad{\v z}aga, Alan Filipin and Yasutsugu Fujita}
\date{\today}
\maketitle
\begin{abstract}
Let $k$ be a positive integer. 
In this paper, we prove that if $\{k,k+1,c,d\}$ is a $D(-k)$-quadruple with $c>1$, then $d=1$. 
\end{abstract}
\noindent
\textrm{2010} \textit{Mathematics Subject Classification}: 11D09, 11B37, 11J68, 11J86

\section{Introduction}\label{intr}

Let $n\neq0$ be an integer. A set of $m$ positive integers $\{a_1,a_2,\ldots,a_m\}$ is called a $D(n)$-$m$-tuple (or a Diophantine $m$-tuple with the property $D(n)$), if $a_ia_j + n$ is a perfect square for all $1\leq i < j \leq m$. 
A natural question regarding such sets is about their possible sizes. 
If $n\equiv2\pmod{4}$, considering congruences modulo 4, it is easy to prove that there does not exist a $D(n)$-quadruple (see for example \cite{Br}, \cite{GS}). 
On the other hand, Dujella \cite{daa} proved that if an integer $n$ does not have the form $4k + 2$ and $n\not \in \{-4, -3, -1, 3, 5, 8, 12, 20\}$, then there exists at least one $D(n)$-quadruple. 
The conjecture is that if $n \in \{-4, -3, -1, 3, 5, 8, 12, 20\}$, then there does not exist a $D(n)$-quadruple. 
In the case $n=-1$, it was proven by Dujella et al.~\cite{DFF} that there does not exist a $D(-1)$-quintuple and that there are only finitely many $D(-1)$-quadruples. 
Also, bounds for the number of $D(-1)$-quadruples have been significantly improved during the years.
But they are still too large to solve the conjecture of the non-existence of $D(-1)$-quadruples completely. 
Furthermore, the most well-known and studied case especially in recent years is when $n=1$, where very recently He, Togb\'e and Ziegler \cite{HTZ} proved the folklore conjecture saying that there does not exist a $D(1)$-quintuple. 
Their result is of great importance because they have introduced some new techniques and ideas. 
However, in the case $n=1$, there is an even stronger conjecture stating that any $D(1)$-triple can be extended to a quadruple with a larger element in a unique way. 
That conjecture is still open, and many mathematicians are working on it. 
For general $n$, if we denote $M_n = \sup\{|S|\}$, where set $S$ has the property $D(n)$, Dujella \cite{dmp}, \cite{dglasnik} proved that $M_n\leq 31$ for $|n|\leq 400$, and $M_n < 15.476 \log|n|$ for $|n| > 400$. 
The whole history of the problem, with recent results and progress, can be found at \cite{web}. \\

In this paper we are interested in the problem of extending a $D(-k)$-pair $\{k,k+1\}$ for a positive integer $k$. 
There are already some results in that direction in the case where $k=K^2$ for a positive integer $K$. 
The third author \cite{Fk2} showed that if $\{K^2, K^2+1, 4K^2+1,d\}$ is a $D(-K^2)$-quadruple, then $d=1$. 
Moreover, the third author and Togb\' e \cite {FT2012} proved, in an elementary and relatively simple manner, that if $\{K^2,K^2 +1,c,d\}$ is a $D(-K^2)$-quadruple with $c < d$, then $c = 1$ and $d = 4K^2 + 1$ (in that case, $3K^2 + 1$ must be a square). 
Similarly, with a few new ideas, the first two authors \cite{AF} have proven that if $\{2K^2,2K^2 + 1,c,d\}$ is a $D(-2K^2)$-quadruple with $c<d$, then $c=1$ and $d=8K^2+1$ (in which case $6K^2+1$ should be a square). 
Our motivation for this paper is to generalize those results for any positive integer $k$. 
However, that problem does not seem so straightforward in general. 
Thus, our main result is the following Theorem:

\begin{thm}\label{thm:main}
Let $k$ be a positive integer.
If $\{k,k+1,c,d\}$ is a $D(-k)$-quadruple with $c>1$, then $d=1$.
\end{thm}

The starting point for the proof of Theorem \ref{thm:main} is the fact that the problem can be reduced to solving the system of Pell (or Pell-like) equations, even in the case where $k$ is non-square. 
While the hypergeometric method (see Theorem \ref{thm:R}) is the main tool to get upper bounds for solutions, which is standard in this area of research, we need a twist to get lower bounds for solutions. 
In fact, it seems hard to get a lower bound for solutions in terms of $k$ by using ``the congruence method'' as in \cite{Dk}, \cite{DP}, \cite{Fk2}. 
Instead, we use the property that the sequences $\{s_{\nu}\}$ and $\{v_{\nu}\}$ attached to $c$ and $d$, respectively, are exactly the same (see \eqref{seq:s} and \eqref{seq:v}), to obtain absolute lower bounds for solutions (see Proposition \ref{prop:n>}). 
Since they are weaker than the ones in terms of $k$, several cases remain to be shown. 
Most of the cases can be done by applying elementary considerations (see Section \ref{sec:k}), using the standard methods, that is, Baker's method on a linear form in logarithms and the reduction method (see Subsections \ref{subsec:78} and \ref{subsec:1}), 
or finding the integral points on certain elliptic curves (see Subsection \ref{subsec:2-5}). 

However, there is a case where we have to find the integral points on a certain hyperelliptic curve of genus $2$ (see Subsection \ref{subsec:6}). 
This task is not easy at all, since Chabauty's method cannot be applied to this curve, 
but we could complete it by applying the strategy due to Gallegos-Ruiz (see \cite{Gall}, \cite{PhD}). 
Such an application is new in this research field, 
and is expected to be utilized in future work. 


We also believe that the following Conjecture is valid, which will be considered in our future research. But as we mentioned, that problem is more difficult for general $k$, than those considered in \cite{AF}, \cite{FT2012}.
\begin{conj}
Let $k$ be a positive integer. If $\{k,k +1,c,d\}$ is a $D(-k)$-quadruple, for positive integer $k$, with $c < d$, then $c = 1$ and $d = 4k + 1$, in which case $3k+1$ must be a square.
\end{conj}

\section{Determination of fundamental solutions}\label{sec:fs}

Let $k$ be a positive integer and $\{k,k+1,c\}$ a $D(-k)$-triple with $c>1$.
Then, there exist positive integers $s'$ and $t$ such that
\begin{align}\label{s't}
  kc-k=(s')^2 \quad \text{and}\quad (k+1)c-k=t^2.
\end{align}
Expressing $k$ as $k=k_0k_1^2$, where $k_0$ and $k_1$ are  positive integers with $k_0$ square-free,
we may write $c-1=k_0s^2$ with some positive integer $s$,
which together with the latter equality of \eqref{s't} implies that
\begin{align}\label{eq:st}
t^2-(k_0^2k_1^2+k_0)s^2=1.
\end{align}
The positive solutions $(t,s)$ to this Pell equation can be expressed as
\[
  t+s\sqrt{k_0^2k_1^2+k_0}=\left(2k+1+2k_1\sqrt{k_0^2k_1^2+k_0}\right)^{\nu},
  \]
which enables us to write $s=s_{\nu}$, where
\begin{align}\label{seq:s}
s_0=0,\quad s_1=2k_1,\quad s_{\nu+2}=(4k+2)s_{\nu+1}-s_{\nu}.
\end{align}
According to $s=s_{\nu}$, we may write $t=t_{\nu}$ and $c=c_{\nu}$.
For the later reference, we list small values of $s_{\nu}$:
\begin{align*}
s_0&=0,\quad s_1=2k_1,\quad s_2=4k_1(2k+1),\\
s_3&=8k_1(2k+1)^2-2k_1,\quad s_4=16k_1(2k+1)^3-8k_1(2k+1),\\
s_5&=32k_1(2k+1)^4-24k_1(2k+1)^2+2k_1,\\
s_6&=64k_1(2k+1)^5-64k_1(2k+1)^3+12k_1(2k+1),\\
s_7&=128k_1(2k+1)^6-160k_1(2k+1)^4+48k_1(2k+1)^2-2k_1,\\
s_8&=256k_1(2k+1)^7-384k_1(2k+1)^5+160k_1(2k+1)^3-16k_1(2k+1),\\
s_9&=512k_1(2k+1)^8-896k_1(2k+1)^6+480k_1(2k+1)^4-80k_1(2k+1)^2+2k_1.
\end{align*}

Since $c_0=1$, if it is proved that there does not exist a $D(-k)$-quadruple $\{k,k+1,c,d\}$ with $1<c<d$, then it turns out that Theorem \ref{thm:main} is valid.
Thus, throughout this paper we assume on the contrary that $\{k,k+1,c,d\}$ is a $D(-k)$-quadruple with $c<d$.
Note that we may consider only the case where $k \ge 3$ in view of \cite{D12} and \cite{Br,GS,MR}.
Then, there exist positive integers $x,\,y,\,z$ such that
\[
  d-1=k_0x^2,\quad (k+1)d-k=y^2,\quad cd-k=z^2,
  \]
from which we obtain the following system of Pellian equations
\begin{align}
y^2-(k_0^2k_1^2+k_0)x^2&=1,\label{P:yx}\\
z^2-k_0cx^2&=c-k.\label{P:zx}
\end{align}
Just as $s=s_{\nu}$, from \eqref{P:yx} we may write $x=v_m$ with a non-negative integer $m$, where $\{v_m\}$ is the recurrence sequence defined by
\begin{align}\label{seq:v}
v_0=0,\quad v_1=2k_1,\quad v_{\nu+2}=(4k+2)v_{\nu+1}-v_{\nu}.
\end{align}
On the other hand, we see from Nagell's argument that for any positive solution $(z,x)$ to \eqref{P:zx} there exist a solution $(z_0,x_0)$ to \eqref{P:zx} satisfying
\begin{align}\label{in:zx}
0<z_0\le \sqrt{c(c-k)},\quad |x_0|<s
\end{align}
and a non-negative integer $n$ such that
\begin{align*}
z+x\sqrt{k_0c}=(z_0+x_0\sqrt{k_0c})(2k_0s^2+1+2s\sqrt{k_0c})^n.
\end{align*}
Thus, we may write $x=w_n$, where
\begin{align}\label{seq:w}
w_0=x_0,\quad w_1=(2k_0s^2+1)x_0+2sz_0,\quad w_{n+2}=(4k_0s^2+2)w_{n+1}-w_n.
\end{align}
Expressions \eqref{seq:s} and \eqref{seq:v} together show that
\[
  (v_m \mod s)_{m \ge 0}=(0,s_1,\dots,s_{\nu-1},0,-s_{\nu-1},\dots,-s_1,0,s_1,\dots),
  \]
which yields $v_m \equiv \pm s_i \pmod{s}$ for some $i$ with $0 \le i < \nu$.
Since we see from \eqref{seq:w} that $w_n \equiv x_0 \pmod{s}$, we have $x_0 \equiv \pm s_i \pmod s$.
It follows from \eqref{in:zx} that $x_0=\pm s_i$.\par
In what follows, we assume the following:
\begin{ass}\label{ass}
$\{k,k+1,c',c\}$ is not a $D(-k)$-quadruple for any $c'$ with $1<c'<c$.
\end{ass}
Then, putting $d_0:=k_0x_0^2+1$ we have
\begin{align*}
kd_0-k&=k_0^2k_1^2x_0^2,\\
(k+1)d_0-k&=(k+1)(k_0x^2+1)-k=k_0(k+1)s_i^2+1=t_i^2,\\
cd_0-k&=c(k_0x_0^2+1)-k=z_0^2,
\end{align*}
that is, $\{k,k+1,d_0,c\}$ is a $D(-k)$-quadruple.
Since $d_0<s^2+1 \le c$, from the assumption we obtain $d_0=1$, i.e., $x_0=0$ and $z_0=\sqrt{c-k}$.
Note that this occurs only if $c-k$ is a perfect square.
Hence, \eqref{seq:w} enables us to express $x=w_n$ as
\begin{align}\label{seq:w'}
w_0=0,\quad w_1=2s\sqrt{c-k},\quad w_{n+2}=(4c-2)w_{n+1}-w_n,
\end{align}
from which we obtain a lower bound for $x$ by $n$ and $c$.
\begin{lem}\label{lem:x>}
If $x=w_n$, then $\log x > (n-1) \log(4c-3)$.
\end{lem}
\begin{proof}
By \eqref{seq:w'} we have
\[
w_n>(4c-3)w_{n-1}>(4c-3)^{n-1}w_1>(4c-3)^{n-1}.
\]
\end{proof}
Moreover, since the recurrence sequence $\{v_m\}$ has the same form as $\{s_{\nu}\}$ and $v_m \equiv \pm s_i =x_0=0 \pmod{s}$, we have the following.
\begin{lem}\label{lem:mod}
If $x=v_m$, then $m \equiv 0 \pmod{\nu}$.
\end{lem}

\section{Lower bounds for solutions}\label{sec:lb}

Our goal in this section is to show the following.
\begin{prop}\label{prop:n>}
Assume that $v_m=w_n$ has a solution with $n \ne 0$.
On Assumption \ref{ass}, the following holds$:$
\begin{itemize}
\item[{\rm (1)}] If $\nu=7$, and $k \ge 12$, then $n \ge 8$.
\item[{\rm (2)}] If either $\nu=8$ and $k \ge 15$ or $\nu \ge 9$ and $k \ge 7$, then $n \ge 9$.
\end{itemize}
\end{prop}

We first consider the case where $n=1$.
It is clear that
\begin{align}\label{in:<w1}
v_{\nu}=s_{\nu}<2s\sqrt{c-k}=w_1.
\end{align}
From sequence \eqref{seq:s} one easily see that
\begin{align}\label{s}
s_{\nu}=\frac{1}{2\sqrt{k_0(k+1)}}\left\{\left(2k+1+2\sqrt{k^2+k}\right)^{\nu}-\left(2k+1-2\sqrt{k^2+k}\right)^{\nu}\right\}.
\end{align}

\begin{lem}\label{lem:n1}
Let $k$ and $l$ be integers with $k \ge 3$ and $l \ge 2$.
\begin{itemize}
\item[{\rm (1)}] $2^lk_0^{l/4}s_{\nu}^l<s_{l \nu}$.
\item[{\rm (2)}] If $l \ge 3$, then $2^lk_0^{l/2}s_{\nu}^l<s_{l \nu}$.
\end{itemize}
\end{lem}
\begin{proof}
(1) If $k=3$ and $l=2$, then we know from \eqref{s} that
\begin{align*}
s_{2\nu}-2^2k_0^{1/2}s_{\nu}^2&=\frac{1}{4\sqrt{3}}\left\{2-2(7-4\sqrt{3})^{2\nu}\right\}>0.
\end{align*}
In all other cases, \eqref{s} implies that
\begin{align*}
s_{l\nu}-2^lk_0^{l/4}s_{\nu}^l&=\frac{1}{2\sqrt{k_0(k+1)}}\left\{\left(2k+1+2\sqrt{k^2+k}\right)^{l\nu}-\left(2k+1-2\sqrt{k^2+k}\right)^{l\nu}\right\}\\
                              &\quad -\frac{1}{k_0^{l/4}(k+1)^{l/2}}\left\{\left(2k+1+2\sqrt{k^2+k}\right)^{\nu}-\left(2k+1-2\sqrt{k^2+k}\right)^{\nu}\right\}^l\\
&>\frac{1}{2k_0^{l/4}(k+1)^{l/2}}\left[\left\{k_0^{(l-2)/4}(k+1)^{(l-1)/2}-2\right\}\left(2k+1+2\sqrt{k^2+k}\right)^{l\nu}\right.\\
                              &\hspace{200pt} -\frac{k_0^{(l-1)/2}(k+1)^{(l-1)/2}}{\left(2k+1+2\sqrt{k^2+k}\right)^{l\nu}}\bigg]\\
&>\frac{1}{2k_0^{l/4}(k+1)^{l/2}}\left[\left\{(k+1)^{(l-1)/2}-2\right\}(4k)^{l\nu}-1\right]>0,
\end{align*}
where the last inequality holds for $k \ge 3$ and $l \ge 2$ with $(k,l)\ne (3,2)$. \par
(2) In the same way as (1), for $l \ge 3$ it holds that
\begin{align*}
s_{l\nu}-2^lk_0^{l/2}s_{\nu}^l&>\frac{1}{2\sqrt{k_0}(k+1)^{l/2}}\left[\left\{(k+1)^{(l-1)/2}-2\sqrt{k_0}\right\}\left(2k+1+2\sqrt{k^2+k}\right)^{l\nu}\right.\\
&\hspace{200pt}
-\frac{(k+1)^{(l-1)/2}}{\left(2k+1+2\sqrt{k^2+k}\right)^{l\nu}}\bigg]\\
&>\frac{1}{2\sqrt{k_0}(k+1)^{l/2}}\left\{(\sqrt{k}-1)^2\left(2k+1+2\sqrt{k^2+k}\right)^{l\nu}-1\right\}>0.
\end{align*}
\end{proof}

We apply Lemma \ref{lem:n1} (1) with $l=2$ to get
\begin{align}\label{in:>w2}
v_{2\nu}=s_{2\nu}>2^2k_0^{1/2}s_{\nu}^2=4s\sqrt{c-1}>2s\sqrt{c-k}=w_1.
\end{align}
Since if $v_m=w_n$, then $m$ is a multiple of $\nu$ by Lemma \ref{lem:mod}, inequalities \eqref{in:<w1} and \eqref{in:>w2} together imply the following.

\begin{lem}\label{lem:w1}
If $k \ge 3$, then $v_m=w_1$ has no solution for all non-negative integers $m$.
\end{lem}

Second, consider the case where $n \ge 2$.

\begin{lem}\label{lem:v>w}
If $k \ge 3$ and $2 \le n \le k+1$, then $v_{2n\nu}>w_n$.
\end{lem}
\begin{proof}
From \eqref{seq:w} we have
\begin{align*}
w_n&<(4c-2)w_{n-1}=(4k_0s^2+2)w_{n-1}<2s\sqrt{c-k}(4k_0s^2+2)^{n-1}\\
   &=2^{2n}k_0^n s^{2n} \cdot \frac{1}{2k_0}\left(1+\frac{1}{2k_0s^2}\right)^{n-1}\sqrt{1-\frac{k-1}{k_0s^2}}.
\end{align*}
Applying Lemma \ref{lem:n1} (2) with $l=2n \ge 4$, we have $2^{2n}k_0^ns^{2n}<s_{2n\nu}=v_{2n\nu}$.
Since $n \le k+1$, it suffices to show
\begin{align}\label{in:<4}
\left(1+\frac{1}{2k_0s^2}\right)^{2n-2}\left(1-\frac{n-2}{k_0s^2}\right) \le 4
\end{align}
for $n \ge 2$.
We show this by induction on $n$. \par
If $n=2$, then \eqref{in:<4} clearly holds.
Assume that \eqref{in:<4} holds for $n$ with $n \ge 2$.
Then,
\begin{align*}
\left(1+\frac{1}{2k_0s^2}\right)^{2n}\left(1-\frac{n-1}{k_0s^2}\right)
&=\left(1+\frac{1}{2k_0s^2}\right)^{2n-2}\left(1-\frac{n-2}{k_0s^2}\right)\left(1+\frac{1}{2k_0s^2}\right)^2\cdot \frac{1-\frac{n-1}{k_0s^2}}{1-\frac{n-2}{k_0s^2}}.
\end{align*}
Since
\begin{align*}
\left(1+\frac{1}{2k_0s^2}\right)^2\cdot \frac{1-\frac{n-1}{k_0s^2}}{1-\frac{n-2}{k_0s^2}}
<\left(1+\frac{1}{2k_0s^2}\right)^2\left(1-\frac{1}{k_0s^2}\right)
=1-\frac{3}{4k_0^2s^4}-\frac{1}{4k_0^3s^6}<1,
\end{align*}
the induction hypothesis shows that
\[
\left(1+\frac{1}{2k_0s^2}\right)^{2n}\left(1-\frac{n-1}{k_0s^2}\right)\le 4.
\]
\end{proof}

\begin{lem}\label{lem:v<w}
Assume that one of the following holds$:$
\begin{itemize}
\item $\nu=7$, $n \le 7$, $k \ge 12$,
\item $\nu=8$, $n \le 8$, $k \ge 15$,
\item $\nu \le 9$, $n \le 8$, $k \ge 7$.
\end{itemize}
Then, $v_{(2n-1)\nu}<w_n$.
\end{lem}
\begin{proof}
We see from \eqref{seq:v} that
\[
v_{(2n-1)\nu}<(4k+2)v_{(2n-1)\nu-1}<2k_1(4k+2)^{(2n-1)\nu-1}
\]
and from \eqref{seq:w'} that
\[
w_n>(4c-3)w_{n-1}=(4k_0s^2+1)w_{n-1}>2s\sqrt{c-k}(4k_0s^2+1)^{n-1}.
\]
Since $s=s_{\nu}>(4k+1)s_{\nu-1}>2k_1(4k+1)^{\nu-1}$ by \eqref{seq:s} and $\sqrt{c-k}>\sqrt{k_0}s>2\sqrt{k}(4k+1)^{\nu-1}$, we have
\begin{align*}
w_n &> 4k_1(4k+1)^{\nu-1}\cdot 2\sqrt{k}(4k+1)^{\nu-1}\left\{16k(4k+1)^{2\nu-2}+1\right\}^{n-1}\\
    &> 8\cdot 16^{n-1}k_1 k^{n-1/2}(4k+1)^{2n(\nu-1)}=2k_1\cdot 4^{n-1/2}(4k)^{n-1/2}(4k+1)^{2n(\nu-1)}\\
    &>2k_1(4k+1)^{2n\nu-n-1/2}.
\end{align*}
Thus, it remains to show the inequality
\[
(4k+1)^{2n\nu-n-1/2}>(4k+2)^{2n\nu-\nu-1},
\]
which is equivalent to
\[
g(\nu,n):=\frac{2n\nu-n-1/2}{2n\nu-\nu-1}>\frac{\log(4k+2)}{\log(4k+1)}=:f(k).
\]
It is easy to check that $g(n,\nu)$ is an increasing function of $\nu$ and a decreasing function of $n$, while $f(k)$ is a decreasing function of $k$.
Since
\begin{align*}
g(7,7)&>1.0055>f(12),\\
g(8,8)&>1.0042>f(15),\\
g(9,8)&>1.0011>f(7),
\end{align*}
we see that if the assumption in the lemma holds, then $g(\nu,n)>f(k)$.
This completes the proof of Lemma \ref{lem:v<w}.
\end{proof}

Now we are ready to prove Proposition \ref{prop:n>}.

\begin{proof}[Proof of Proposition $\ref{prop:n>}$]
We may assume that $n \ge 2$ in view of Lemma \ref{lem:w1}.
Suppose that $\nu$, $n$, $k$ satisfy one of conditions in Lemma \ref{lem:v<w}.
Since $m \equiv 0 \pmod{\nu}$ by Lemma \ref{lem:mod}, it suffices to show that
\[
v_{(2n-1)\nu}<w_n<v_{2n\nu}.
\]
This is an immediate consequence of Lemmas \ref{lem:v>w} and \ref{lem:v<w}.
\end{proof}

\section{Upper bounds for solutions}\label{sec:ub}

Put
\[
\theta_1:=\sqrt{1-\frac{1}{c}}\quad \text{and}\quad \theta_2:=\sqrt{1-\frac{k}{(k+1)c}}.
\]

\begin{lem}\label{lem:max<}
\[
\max\left\{\left| \theta_1-\frac{k_1sx}{z}\right|,\left|\theta_2-\frac{ty}{(k+1)z}\right|\right\}<\frac{1}{2k_0x^2}.
\]
\end{lem}
\begin{proof}
By \eqref{P:yx} and \eqref{P:zx}, we have
\begin{align*}
\left|\theta_1-\frac{k_0sx}{z}\right|&=\frac{s\sqrt{k_0}}{z\sqrt{c}}\left|z-x\sqrt{k_0c}\right|=\frac{s\sqrt{k_0}}{z \sqrt{c}}\cdot \frac{c-k}{z+x\sqrt{k_0c}}\\
&<\frac{1}{x\sqrt{k_0c}}\cdot \frac{c}{2x\sqrt{k_0c}}=\frac{1}{2k_0x^2},\\
\left|\theta_2-\frac{ty}{(k+1)z}\right|&=\frac{t}{(k+1)z\sqrt{c}}\left|z\sqrt{k+1}-y\sqrt{c}\right|=\frac{t}{(k+1)z\sqrt{c}}\cdot \frac{k(c-k-1)}{z\sqrt{k+1}+y\sqrt{c}}\\
&<\frac{kc}{y\sqrt{c}\cdot 2y\sqrt{c}}=\frac{k}{2y^2}=\frac{k}{2\left\{k_0(k+1)x^2+1\right\}}<\frac{1}{2k_0x^2}.
\end{align*}
\end{proof}

\begin{thm}\label{thm:R}
Let $k \ge 3$ and let $N$ be a multiple of $k+1$.
If $N \ge 3.76 k^2(k+1)^2$, then
\[
\max\left\{\left|\theta_1-\frac{p_1}{q}\right|,\left|\theta_2-\frac{p_2}{q}\right|\right\}>(1.425\cdot 10^{28}(k+1)N)^{-1}q^{-\lambda},
\]
where
\[
\lambda=1+\frac{\log(10(k+1)N)}{\log(2.66k^{-2}(k+1)^{-1}N^2)}<2.
\]
\end{thm}
\begin{proof}
The proof proceeds along the same lines as the one of \cite[Theorem 2.2]{CF} or \cite[Theorem 2.5]{FF}.\par
For $0 \le i,\,j \le 2$ and integers $a_0$, $a_1$, $a_2$,
we define the polynomial $p_{ij}(x)$ by
\begin{align*}
p_{ij}(x):=\sum_{ij}\left(\begin{array}{c}k+1/2\\ h_j\end{array}\right)
(1+a_jx)^{k-h_j}x^{h_j}\prod_{l \ne j}\left(\begin{array}{c}-k_{ij}\\ h_l\end{array}\right)(a_j-a_l)^{-k_{il}-h_l},
\end{align*}
where $k_{il}=k+\delta_{il}$ with $\delta_{il}$ the Kronecker delta,
$\sum_{ij}$ denotes the sum over all non-negative integers
$h_0$, $h_1$, $h_2$ satisfying $h_0+h_1+h_2=k_{ij}-1$, and
$\prod_{l\ne j}$ denotes the product from $l=0$ to $l=2$
omitting $l=j$ (which is expression (3.7) in \cite{R} with $\nu=1/2$).
Substituting $x=1/N$ we have
\[
p_{ij}(1/N)=\sum_{ij}\left(\begin{array}{c}k+1/2\\ h_j\end{array}\right)
C_{ij}^{-1}\prod_{l \ne j}\left(\begin{array}{c}-k_{ij}\\ h_l\end{array}\right),
\]
where
\[
C_{ij}:=\frac{N^k}{(N+a_j)^{k-h_j}}\prod_{l \ne j}(a_j-a_l)^{-k_{il}-h_l}.
\]
We take $a_0:=-k-1$, $a_1:=-k$, $a_2:=0$ and $N:=(k+1)N_0$ for some integer $N_0$. \par
If $j=0$, then
\[
|C_{i0}|=\frac{N^k(k+1)^{k_{i1}+h_1+h_0-k}}{(N_0-1)^{k-h_0}},
\]
which shows $(k+1)^kN^kC_{i0}^{-1} \in \bZ$.
If $j=1$, then
\[
|C_{i1}|=\frac{N^kk^{k_{i2}+h_2}}{(N-k)^{k-h_1}},
\]
which shows $k^{2k}N^kC_{i1}^{-1} \in \bZ$.
If $j=2$, then
\[
|C_{i2}|=\frac{N^k(k+1)^{k_{i0}+h_0+h_2-k}k^{k_{i1}+h_1}}{N_0^{k-h_2}},
\]
which shows $k^{2k}(k+1)^kN^kC_{i2}^{-1} \in \bZ$.
Hence, $\{k^2(k+1)N\}^kC_{ij}^{-1} \in \bZ$ for all $i$, $j$.
It follows from \cite[Theorem 2.2]{CF} that
\[
p_{ijk}:=2^{-1}\{4k^2(k+1)N\}^k\cdot \frac{1.6^k}{4.09\cdot 10^{13}}\cdot p_{ij}(1/N)\in \bZ.
\]
Putting $\theta_0:=1$, we obtain
\[
|p_{ijk}|<pP^k\quad \text{and}\quad \left|\sum_{j=0}^2 p_{ijk}\theta_j \right|<lL^{-k},
\]
where
\begin{align*}
p&=\frac{4.09\cdot10^{13}}{2}\left(1+\frac{k}{2N}\right)^{1/2}<2.048\cdot 10^{13},\\
P&=\frac{32\left(1+\frac{2k+3}{2N}\right)k(k+1)N}{1.6(2k+1)}<10(k+1)N,\\
l&=\frac{4.09\cdot10^{13}}{2}\cdot \frac{27}{64}\left(1-\frac{k+1}{N}\right)^{-1}<8.692\cdot10^{12},\\
L&=\frac{1.6}{4k^2(k+1)N}\cdot \frac{27}{4}\left(1-\frac{k+1}{N}\right)^2N^3>\frac{2.66N^2}{k^2(k+1)},\\
\lambda&=1+\frac{\log(10(k+1)N)}{\log(2.66k^{-2}(k+1)^{-1}N^2)}<2,\\
C^{-1}&=4pP\left(\max\{1,2l\}\right)^{\lambda-1}<4\cdot 2.048\cdot10^{13}\cdot10(k+1)N\cdot2\cdot8.692\cdot10^{12}\\
&<1.425\cdot10^{28}(k+1)N.
\end{align*}
This completes the proof of Theorem \ref{thm:R}.
\end{proof}

Applying Theorem \ref{thm:R} with $N=(k+1)c$, $p_1=k_0(k+1)sx$, $p_2=ty$, $q=(k+1)z$ and Lemma \ref{lem:max<}, we have
\[
\left(1.425\cdot 10^{28}(k+1)^2c\right)^{-1}(k+1)^{-\lambda}z^{-\lambda}<\frac{1}{2k_0x^2}.
\]
Since $z^2=k_0cx^2+c-k<k_0(c+1)x^2$ by $c \le d-1=k_0x^2$, we see from $\lambda<2$ that
\begin{align*}
x^{2-\lambda}&<\frac12 \cdot 1.425\cdot10^{28}(k+1)^4c(c+1)=7.125\cdot10^{27}\left(1+\frac1k\right)^4\left(1+\frac1c\right)k^4c^2\\
&<\left(1.502\cdot10^{14}k^2c\right)^2.
\end{align*}
Since
\begin{align*}
\frac{2}{2-\lambda}&=\frac{2\log \left(2.66k^{-2}(k+1)c^2 \right)}{\log\left(0.266k^{-2}(k+1)^{-1}c\right)}<\frac{4\log(1.884k^{-1/2}c)}{\log(0.1995k^{-3}c)},
\end{align*}
we have
\begin{align*}
\log x < \frac{4\log(1.502\cdot 10^{14}k^2c)\log(1.884k^{-1/2}c)}{\log(0.1995k^{-3}c)}.
\end{align*}
which combined with Lemma \ref{lem:x>} implies that
\begin{align}\label{in:n<}
n-1<\frac{4\log(1.502\cdot10^{14}k^2c)\log(1.884k^{-1/2}c)}{\log(4c-3)\log(0.1995k^{-3}c)}.
\end{align}
Inequality \eqref{in:n<} shows that if $\nu=8$ and $k \ge 662$, then $n \le 7$,
and if $\nu=8$ and $k \ge 5$, then $n \le 8$.
Moreover, since the right-hand side of inequality \eqref{in:n<} is a decreasing function of $c=c_{\nu}$,
we see that if $\nu \ge 9$ and $k \ge 3$, then $n \le 8$.
Comparing these upper bounds for $n$ with the lower bounds in Proposition \ref{prop:n>}, we obtain the following.
\begin{prop}\label{prop:ub}
Besides Assumption \ref{ass}, we assume that one of the following holds$:$
\begin{itemize}
\item $\nu=7$ and $k \ge 662$,
\item $\nu=8$ and $k \ge 15$,
\item $\nu \ge 9$ and $k \ge 7$.
\end{itemize}
Then, there exist no $D(-k)$-quadruples of the form $\{k,k+1,c,d\}$ with $c=c_{\nu}$ and $1<c<d$.
\end{prop}

In view of Proposition \ref{prop:ub}, it remains to consider the following cases:
\begin{itemize}
\item $k=3,~5,~6$,
\item $1 \le \nu \le 6$ and $k \ge 7$,
\item $\nu=7$ and $7 \le k \le 661$,
\item $\nu=8$ and $7 \le k \le 14$.
\end{itemize}

\section{Linear form in logarithms}
We are trying to solve $x=v_m=w_n$, where
\[ v_0 = 0, v_1 = 2k_1, v_{m+2} = (4k+2)v_{m+1}-v_m, \]
($k=k_0k_1^2$). The solution of this recurrence relation is
\[ v_m = \frac{1}{2\sqrt{k_0(k+1)}}\left( (2k+1+2\sqrt{k^2+k})^m - (2k+1-2\sqrt{k^2+k})^m\right).\]

\noindent The other sequence is $w_0 = 0, w_1=2s\sqrt{c-k}, w_{n+2} = (4c-2)w_{n+1}-w_n$,
or explicitly \[ w_n = \frac{\sqrt{c-k}}{2\sqrt{ck_0}}\left( (2c-1+2\sqrt{c^2-c})^n - (2c-1-2\sqrt{c^2-c})^n\right).\]

\noindent Lemma 3.3 
implies that $m> n \geqslant 2$ or $x=v_0=w_0=0$, so we assume $k\geqslant 3$ and $m>n\geqslant 2$.

Define $P=\frac{1}{\sqrt{k+1}}(2k+1+2\sqrt{k^2+k})^m$ and $Q=  \frac{\sqrt{c-k}}{\sqrt{c}}(2c-1+2\sqrt{c^2-c})^n$. Then $v_m=w_n$ implies that $P-\frac{1}{k+1}P^{-1} = Q-\frac{c-k}{c}Q^{-1}$. Since $c\geqslant c_1 = 4k+1$, we get $\frac{c-k}{c} > \frac{1}{k+1}$.
Then
\begin{align*}
P - Q &= \frac{1}{k+1}P^{-1} - \frac{c-k}{c}Q^{-1} <  \frac{1}{k+1}P^{-1} -  \frac{1}{k+1}Q^{-1} =\\
&=  \frac{1}{k+1}(P^{-1}-Q^{-1}) = \frac{1}{k+1}P^{-1}Q^{-1} (Q-P),
\end{align*}
hence $Q>P$.

Now $Q-P < \frac{c-k}{c}Q^{-1}$, hence $\frac{Q-P}{Q} < \frac{c-k}{c}Q^{-2}$. Since $n\geqslant 2$, $Q > \sqrt{\frac{c-k}{c}}(4\sqrt{c^2-c})^2 = 16(c^2-c)\sqrt{\frac{c-k}{c}}$, so $Q^{-2} < \frac{c}{c-k}\cdot \frac{1}{256(c^2-c)^2}$. We can conclude that $\frac{Q-P}{Q} < \frac{1}{256(c^2-c)^2} \leqslant \frac{1}{6230016}$ because $c\geqslant 1+4k \geqslant 13$.

We can get the upper bound on $\log\frac{Q}{P}$,
\begin{align*}
0 < \log\frac{Q}{P} &= -\log\left(1-\frac{Q-P}{Q}\right) \leqslant \frac{-\log\left(1-\frac{1}{6230016}\right)}{\frac{1}{6230016}} \cdot \frac{c-k}{c} Q^{-2} \\
&< 1.00001 (2c-1+2\sqrt{c^2-c})^{-2n}.
\end{align*}

We define the form as $\Lambda = n\log\alpha_1 - m\log\alpha_2+\log\alpha_3$, where
\begin{align*}
\alpha_1 &= 2c-1+2\sqrt{c^2-c},& & h(\alpha_1) = \frac 12 \log\alpha_1\\
\alpha_2 &= 2k+1+2\sqrt{k^2+k},& & h(\alpha_2) = \frac 12 \log\alpha_2\\
\alpha_3 &= \sqrt{\frac{(c-k)(k+1)}{c}},& & h(\alpha_3) = \frac 12 \log\left((c-k)(k+1)\right).
\end{align*}
 and $h(\alpha_j)$ denotes the absolute logarithmic height of $\alpha_j$ for $1\leqslant j \leqslant 3$.

We have already obtained $0 < \Lambda < 1.0001\alpha_1^{-2n}$. This implies that $m < \frac{\log\alpha_1}{\log\alpha_2}(n+1)$. We will now apply the following theorem by Matveev.

\begin{theorem}[Matveev]
Let $\Lambda$ be a linear form in logarithms of $l$ multiplicatively independent totally real algebraic numbers $\alpha_1, \dotsc, \alpha_l$ with rational integer coefficients $b_1, \dotsc, b_l$ ($b_l \neq 0$). Define $D=[\Q(\alpha_1, \dotsc, \alpha_l) \colon \Q]$, $A_j = \max\{Dh(\alpha_j), |\log \alpha_j|\}$, $B= \max\left\{1, \max\left\{ \frac{|b_j|A_j}{A_l} \colon 1 \leqslant j \leqslant l\right\} \right\}.$ Then
\[ \log\Lambda > -C(l)C_0W_0D^2\Omega,\]
where $C(l)= \frac{8}{(l-1)!}(l+2)(2l+3)(4e(l+1))^{l+1}, W_0 = \log(1.5eBD\log(eD))$, \\$C_0 = \log\left(e^{4.4l+7}l^{5.5}D^2\log(eD)\right), \Omega=A_1\cdots A_l$.
\end{theorem}

In our problem, $l=3, b_1=n, b_2=-m, b_3=1, D=4$. Since $B < m <  \frac{\log\alpha_1}{\log\alpha_2}(n+1)$, we can get the following bounds
$C(3) < 644065984.903$, $C_0 < 29.8847$, $W_0 < \log\left(38.92\cdot \frac{\log\alpha_1}{\log\alpha_2}(n+1)\right)$.
For $\Omega$ we can take $\Omega = 8\log\alpha_1\log\alpha_2\log\left((c-k)(k+1)\right)$.

Combining the upper and lower bound for $\log\Lambda$ and using $\alpha_2 < 4k+2$, we get
\begin{align}\label{in:log}
\frac{n}{\log\left(K(n+1)\right)} < 1.23185\cdot 10^{12} \log(4k+2)\log\left((c-k)(k+1)\right),
\end{align}
where $K=38.92 \frac{\log\alpha_1}{\log\alpha_2}$.

\section{Small values of $k$}\label{sec:k}

In the case $k=3$, we have $$c-k=1+3s^2-3=3s^2-2,$$ which should be a square. We also know that $s$ is even, i.e. $s=2s'$ for some integer $s'$. Then, putting $3s^2-2=X^2$, for some integer $X$, we get $$X^2-12s'^2=-2,$$ which obviously does not have any integer solutions if we consider congruences modulo 4. \\

In the case $k=5$, we have that $$c-k=5s^2-4$$ is perfect square. Again, because $s$ is even, for $s=2s'$, we get $$X^2-5s'^2=-1,$$ for some integers $X$ and $s'$. Then, remembering the recurrence relation for $s$, we get that $s=v_m=2w_n$ where $$v_0=0,\,v_1=2,\,v_{m+2}=22v_{m+1}-v_m,$$  $$w_0=1,\,w_1=17,\,w_{n+2}=18w_{n+1}-w_n.$$ Using the standard methods, i.e. Baker's theory on a linear form in logarithms, we get that the only solution is $s=v_1=2w_0=2$, which gives us $c=21$. Now, we have the exact values for $k$ and $c$ and we can again use the linear form in logarithms described above, and it will gives us the desired result. \\

In the case $k=6$, we have that $-k\equiv2\pmod{4}$ and then it is known that there is no $D(-6)$-quadruple.

\section{Small indices $\nu \in \{1, 2, 3, 4, 5, 6, 7, 8\}$ }\label{sec:2-8}

\subsection{Case $\nu \in \{7,8\}$}\label{subsec:78}

Whenever we have fixed $k$ and $c=c_{\nu}$, we can solve our problem using inequalities \eqref{in:log}. 
After getting the first upper bound on $n$, we can reduce it using the well-known Baker-Davenport reduction method, which gives us the desired result in all cases. 
More precisely, in the case $c=c_7$, for $7\leq k\leq 661$ we get $n<4.73\cdot10^{16}$, 
and in the case $c=c_8$, for $7\leq k \leq14$ we get $n<1.39\cdot10^{16}$. 
Using the reduction, after at most two steps we get $n\leq2$ and we can check that the only solution for $x$ is $x=v_0=w_0=0$ which gives us $c=1$.

\subsection{Case $\nu =1$}\label{subsec:1}

For $\nu=1$ we have $c=c_1=1+4k$, $k\geq7$. Now, $c-k=1+3k$ should be a square which implies $1+3k=(3l\pm1)^2$ or $k=l(3l\pm2)$ for some positive integer $l$. Thus, $\sqrt{c-k}=3l\pm1$. In this case $s=s_1=2k_1$ and we want to solve $v_m=w_n$, for positive integers $m$ and $n$, where
$$v_0=0,\, v_1=2k_1,\,v_{m+2}=(4k+2)v_{m+1}-v_m,$$
$$w_0=0,\, w_1=4k_1(3l\pm1),\,w_{n+2}=(16k+2)w_{n+1}-w_n.$$
Considering congruences modulo $2k_1(4k+1)$ we get
$$v_m\equiv0,\pm2k_1\pmod{2k_1(4k+1)},$$
$$w_n\equiv(-1)^{n+1}4nk_1(3l\pm1)\pmod{2k_1(4k+1)},$$ which implies
$$(-1)^{n+1}4nk_1(3l\pm1)\equiv0,\pm2k_1\pmod{2k_1(4k+1)}.$$
Now, from $$3(4k+1)=(12l\pm4)(3l\pm1)-1,$$ we see that $(3l\pm1)$ and $(4k+1)$ are relatively prime and then $$\pm2n\equiv0,\pm(12l\pm4)\pmod{(4k+1)},$$ which, in the worst case, implies that $$2n>2(4k+1)-12l-4>4k.$$

Assuming $n\geq2$, we can combine this lower bound for $n$ with the upper bound \eqref{in:log} to get $k<8.528\cdot10^{16}$ and finally $l<1.68603\cdot10^8$, which is small enough to do the Baker-Davenport reduction method, which gives us the desired result in the same way as in the last subsection.

\subsection{Cases $\nu \in \{2, 3, 4, 5\}$}\label{subsec:2-5}

For $\nu = 2$ we get $y^2 = c_2-k = k_0s_2^2+1-k= 64k^3+64k^2+15k+1$.
This is an elliptic curve. Multiplying by $64$ and letting $X=16k, Y=8y$, we get Weierstrass form $Y^2 = X^3+16X^2+60X+64$.
Using Sage to find integral points on this curve and condition $16\mid X$ gives that $64k^3+64k^2+15k+1$ is a square only for $k=0$ and $k=1$ ($1^2$ and $12^2$).

For $c_3$ we get
$y^2 = c-k = 1 + 35 k + 384 k^2 + 1408 k^3 + 2048 k^4 + 1024 k^5$,
but the right-hand side can be factored to $(1 + 16 k + 32 k^2) (1 + 19 k + 48 k^2 + 32 k^3)$.
The factors are relatively prime, so $1 + 19 k + 48 k^2 + 32 k^3$ must also be a square. This gives us another elliptic curve and as before, one gets only a few points on it (for $k=0, 1$ and $165$).

For $c_4$, $y^2 = 1+63k + 1280k^2+9472k^3 +32768k^4+57344k^5 + 49152k^6+16384k^7$
factors as $y^2 = (1 + 32 k + 128 k^2 + 128 k^3) (1 + 31 k + 160 k^2 + 256 k^3 + 128 k^4)$,
so $Y^2 = 1 + 32 k + 128 k^2 + 128 k^3$. Only nonnegative integral $k$ are $k=0$ and $k=1$.

For $c_5$, $y^2 = (1 + 48 k + 352 k^2 + 768 k^3 + 512 k^4) (1 + 51 k + 400 k^2 +
   1120 k^3 + 1280 k^4 + 512 k^5)$, so $Y^2=1 + 48 k + 352 k^2 + 768 k^3 + 512 k^4$. Using IntegralQuarticPoints([512, 768, 352, 48, 1]) in Magma, we get that the only integral solutions are $(1, -41), (-1, 7 ), ( 0, -1 )$.

\subsection{Case $\nu = 6$: hyperelliptic curve of genus $2$}\label{subsec:6}
   For $\nu=6$, we get \[ y^2=(1 + 72 k + 768 k^2 + 2816 k^3 + 4096 k^4 + 2048 k^5) (1 + 71 k +
   840 k^2 + 3584 k^3 + 6912 k^4 + 6144 k^5 + 2048 k^6).\]

   The factors on the right hand side are relatively prime, so each one has to be a complete square. We focus on the first factor. If  $1 + 72 k + 768 k^2 + 2816 k^3 + 4096 k^4 + 2048 k^5$ is a square, then so is $16(1 + 72 k + 768 k^2 + 2816 k^3 + 4096 k^4 + 2048 k^5)=16+\dots+(8k)^5$. This allows us to make the coefficients smaller by a change of variable $x=8k$ and multiplying $y$ by $4$. Now, we are looking for integral points on the following hyperelliptic curve \[ C_6 : y^2 =x^5 + 16x^4 + 88x^3 + 192x^2 + 144x +16.\]
   We resolve this problem using methods developed by Gallegos-Ruiz in his PhD thesis \cite{PhD} and in \cite{Gall}.

   Using Magma \cite{mgma}, one can determine generators for the Mordell-Weil group of $J_6 (\Q)$, the Jacobian of $C_6$. We obtain that $J_6(\Q)$ is free of rank $r=2$ with Mordell-Weil basis (written in Mumford representation that Magma uses):
\[ D_1 = \langle x+2, -4\rangle, \quad D_2 = \langle x^2 + 8x + 12, 4\rangle, \]
   while the torsion subgroup is trivial (so we let $t=1$, the size of the torsion subgroup).

   Baker's method, improved in \cite{Gall}, gives us a very large bound $\log |x| \leqslant 1.53106\cdot 10^{489}$.
    Every integral point $P$ on the curve $C_6$ can be expressed as $P-\infty = n_1D_1+n_2D_2$ with norm $||(n_1, n_2)|| \leqslant 1.2203552 \cdot 10^{245} =:N$, by the Corollary 3.2 of \cite{Gall}. Proposition 6.2 from the same paper gives us an estimate of the precision we need for the computations that will follow. This bound is $\displaystyle \left(\frac 15(48\sqrt{r}Nt+12\sqrt{r}N+5N+48)\right)^{(r+4)/4} \approx 3.25\cdot 10^{369}$. We need a constant $K$ larger than this and we chose $K=10^{750}$, and the computations were done with $1000$ digits of precision.

     The hyperelliptic logarithms of the base divisors are given by
     \begin{align*}
     \varphi (D_1) &= (-0.57355...-i1.292539..., -0.337441...+i0.979713...) \\
      \varphi(D_2) &= (-0.09728...+i0.691157..., -0.40469...-i2.809269...).
      \end{align*}
  The chosen $K$ reduces the bound on the norm of the coefficients to $129.97...$. We then repeat the reduction process with $K=10^{10}$ and this reduces the bound on $||(n_1, n_2)||$ to $17.9141...$, which is sufficiently low for the simple search.  Now we just compute all possible expressions of the form $n_1D_1+n_2D_2$ where $||(n_1, n_2)|| \leqslant 17.92$. This shows that the only integral points on the curve $C_6$ are \[ \{ \infty, (0, -4), (0, 4), (-2, -4), (-2, 4), (-6, -4), (-6, 4), (8, -396), (8, 396)\}. \]
  Since $x=8k$, returning to the original factor $1 + 72 k + 768 k^2 + 2816 k^3 + 4096 k^4 + 2048 k^5$, we see that it is a square only for $k=0, 1$.

\section*{Acknowledgements}
N.~A.~and A.~F.~were supported by the Croatian Science Foundation under the project no.~IP-2018-01-1313.

Y.~F.~is supported by JSPS KAKENHI Grant Number 16K05079.


%
Faculty of Civil Engineering, University of Zagreb, Fra Andrije
Ka\v{c}i\'{c}a-Mio\v{s}i\'{c}a 26, 10000 Zagreb, Croatia \\
Email: nadzaga@grad.hr \\[6pt]
Faculty of Civil Engineering, University of Zagreb, Fra Andrije
Ka\v{c}i\'{c}a-Mio\v{s}i\'{c}a 26, 10000 Zagreb, Croatia \\
Email: filipin@grad.hr \\[6pt]
Department of Mathematics, College of Industrial Technology, Nihon
University, 2-11-1 Shin-ei, Narashino, Chiba, Japan \\
Email: fujita.yasutsugu@nihon-u.ac.jp\\[6pt]

\end{document}